\documentclass[final,3p,times]{elsarticle}
\usepackage{amsmath}
\allowdisplaybreaks[4]
\usepackage{amsthm}
\usepackage{amscd}
\usepackage{amsfonts}
\usepackage{amssymb}
\usepackage{graphicx}

\newtheorem{theorem}{Theorem}[section]

\newtheorem{lemma}[theorem]{Lemma}

\usepackage{mathrsfs}
\usepackage{titletoc}
\numberwithin{equation}{section}

\newcommand{\D}{\Delta}
\newcommand{\ra}{\rightarrow}

\newcommand{\f}{\frac}

\renewcommand{\l}{\lambda}
\renewcommand{\L}{\Lambda}
\newcommand{\be}{\begin{equation}}
\renewcommand{\ra}{\rightarrow}
\newcommand{\ee}{\end{equation}}
\newcommand{\bea}{\begin{eqnarray}}
\newcommand{\eea}{\end{eqnarray}}
\newcommand{\bna}{\begin{eqnarray*}}
\newcommand{\ena}{\end{eqnarray*}}
\renewcommand{\o}{\omega}

\newcommand{\iv}{\int_{V}}

\newcommand{\me}{\mathrm{e}}

\renewcommand{\le}{\left}
\newcommand{\ri}{\right}

\newcommand{\ve}{\vert}
\newcommand{\V}{\Vert}

\newcommand{\na}{\nabla}

\newcommand{\up}{\upsilon}

\journal{***}
\begin{document}
   
\begin{frontmatter}

\title{ Existence of solutions to  a generalized  self-dual Chern-Simons system on finite graphs}

\author{Ruixue Chao$^a$}
\ead{2019309040125@cau.edu.cn}

\author{Songbo Hou$^b$ \corref{cor1}}
\ead{housb@cau.edu.cn}
\author{Jiamin Sun$^c$}
\ead{1416525364@qq.com}

\address{$^{a,b,c}$Department of Applied Mathematics, College of Science, China Agricultural University,  Beijing, 100083, P.R. China}
\cortext[cor1]{Corresponding author: Songbo Hou}

\begin{abstract}
We study a system of equations arising in the Chern-Simons model on finite graphs. Using the iteration scheme and the upper and lower solutions method,  we get existence of solutions in the non-critical case. The critical case is dealt with by priori estimates. Our results generalize those of Huang et al. (Journal of Functional Analysis 281(10) (2021) Paper No. 109218).
\end{abstract}
\end{frontmatter}
 %\tableofcontents

\section{Introduction}

The Chern-Simons models describe gauge fields  governed by  Chern-Simons type dynamics, and explain certain phenomena in the fields of particle physics, condensed matter physics and so on \cite{kumar1986charged, zwanziger1968exactly, jain1989composite}.  Some Chern-Simons models can be reduced to elliptic equations with exponential nonlinearities. Many studies were devoted to self-dual Chern-Simons equations including nonrelativistic and relativistic cases, Abelian
and non-Abelian cases.

In this paper, we consider the following Chern-Simons system

\be\label{uve}
\left\{\begin{aligned}
\Delta u &=-\lambda\me^{\up}H(\me^{\up})g(\me^{u})+4\pi\sum\limits_{j=1}^{N_1}\delta_{p_j'},\\
\Delta \up&=-\lambda\me^{u}G(\me^{u})h(\me^{\up})+4\pi\sum\limits_{j=1}^{N_2}\delta_{p_j^{''}},
\end{aligned}
\right.
\ee
on a finite graph, where $G> 0$, $H>0$ are increasing, $C^{\infty}$ functions in $[0,\infty)$; $g$ and $h$ are defined by
  $g(s^2)=\int_s^1 2sG(s^2)$ ds and $h(s^2)=\int_s^1 2sH(s^2)ds$  respectively; $\l>0$ is a constant; $N_1$ and $N_2$ are positive integers;
$\delta_{p}$ is  the Dirac delta mass at vertex $p$. The system (\ref{uve}) was proposed in \cite{nam2013vortex} to study the $U(1)\times U(1)$ Chern-Simons model with a general Higgs potential. For the special case $G\equiv1$ and $H\equiv1$, the existence of solutions to the system (\ref{uve}) was obtained in \cite{lin2007system,lin2009vortex}, and the discrete form of (\ref{uve}) on finite graphs was investigated in \cite{huang2021mean}. For more results on discrete equations with exponential nonlinearities, one  may refer to \cite{gao2022existence,  ge2020p, ge2018kazdan, grigor2016kazdan, hou2022existence, liu2020multiple, lu2021existence, zhang2018p}.

 We write $G=(V,E)$ to denote a connected finite graph, where $V$ and $E$ represent vertices and edges respectively.
 We assume the weight $\o_{xy}>0$ on edge $xy$ is symmetric.
  Let $\mu :V\ra \mathbb{R}^{+}$ be a finite measure. For functions $u,\up:V\ra \mathbb{R}$, we define the
$\mu$-Laplace operator   by
\be\label{mul}\D u(x)=\f{1}{\mu(x)}\sum\limits_{y\sim x}\omega_{xy}(u(y)-u(x)),\ee
and let 
\be\label{lpd}\Gamma(u,\up)=\f{1}{2\mu(x)}\sum\limits_{y\sim x}\omega_{xy}(u(y)-u(x))(\up(y)-\up(x)),\ee
where  $y\sim x$ means vertex $y$ is adjacent
to vertex $x$.
Write 
$$\ve \na u\ve (x) =\le( \f{1}{2\mu(x)}\sum\limits_{y\sim x}\omega_{xy}\le(u(y)-u(x)\ri)^2\ri)^{\f{1}{2}}.$$
For any  function $f:V\ra \mathbb{R}$,  the integral of $f$  over $V$ is defined by
$$\int_V fd\mu=\sum\limits_{x\in V}\mu(x)f(x).$$
We define the Sobolev space as in the Euclidean case by  $$W^{1,2}(V)=\le\{u\,\Big|\, u:V\rightarrow \mathbb{R},\,\iv\le(\ve \na u\ve^2 +u^2\ri)d\mu<+\infty\ri\}.$$

We get the following results about the existence of maximal solutions. 
\begin{theorem}
There exists
$\l_c \geq \f{4\pi\max\{N_1,N_2\}}{G(1)H(1)\ve V\ve }$
such that
\begin{enumerate}[(i)]
\item  If $\l>\l_c$, the system (\ref{uve}) admits a unique maximal solution $(u_{\l},\up_{\l})$ in the sense that if $(u_{\l}',\up_{\l}')$ is any other
solution,  then $u_{\l}>u_{\l}'$, $\up_{\l}>\up_{\l}'$. Moreover, if $\l_1>\l_2>\l_c$, then $u_{\l_1}>u_{\l_2}$
 and $\up_{\l_1}>\up_{\l_2}$.

\item   If $\l<\l_c$, the system (\ref{uve}) admits no solution.

\item  If $\l=\l_c$, the system (\ref{uve}) admits a solution $(u_{*}, \up_{*})$ which satisfies $u_{*}<u_{\l}$ and
$\up_{*}<\up_{\l}$ if $\l_c<\l$.
\end{enumerate}

\end{theorem}
 We also use the iterative scheme as in \cite{lin2009vortex,nam2013vortex, huang2020existence}, while use  different methods in  the proof of the case (iii) in Theorem 1.1. Our results generalize those of Huang et al. \cite{huang2020existence}.

\section{Proof of the main results}
Let $(u_0, \up_0)$  be the solution to the system

\be\left\{\begin{aligned}
\Delta u &=-\f{4\pi N_1}{\ve V \ve}+4\pi\sum\limits_{j=1}^{N_1}\delta_{p_j'},\\
\Delta \up&=-\f{4\pi N_2}{\ve V \ve}+4\pi\sum\limits_{j=1}^{N_2}\delta_{p_j^{''}}.
\end{aligned}
\right.
\ee
 Set $u'=u_0+u$ and $\up'=\up_0+\up$  if $(u',\up')$ is the solution to the system (\ref{uve}). Substituting them into (\ref{uve}) gives
\be\label{uvs}
\left\{\begin{aligned}
\Delta u &=-\l\me ^{\up_0+\up}H(\me^{\up_0+\up})g(\me^{u_0+u})+\f{4\pi N_1}{\ve V\ve},\\
\Delta \up&=-\l\me ^{u_0+u}G(\me^{u_0+u})h(\me^{\up_0+\up})+\f{4\pi N_2}{\ve V\ve}.
\end{aligned}
\right.
\ee
We say that $(u_-,\up_-)$ is a lower solution of (\ref{uvs}) if it satisfies 
\be\label{lsd}
\left\{\begin{aligned}
	\Delta u_- &\geq -\l\me ^{\up_0+\up_-}H(\me^{\up_0+\up_-})g(\me^{u_0+u_-})+\f{4\pi N_1}{\ve V\ve},\\
	\Delta \up_-&\geq-\l\me ^{u_0+u_-}G(\me^{u_0+u_-})h(\me^{\up_0+\up_-})+\f{4\pi N_2}{\ve V\ve}.
\end{aligned}
\right.
\ee
Let $(u_1,\up_1)=(-u_0,-\up_0)$.  We carry out  the following iteration procedure
\be\label{uvi}
\left\{\begin{aligned}
\le(\Delta-K\ri) u_{n+1} &=-\l\me ^{\up_0+\up_n}H(\me^{\up_0+\up_n})g(\me^{u_0+u_n})-Ku_n+\f{4\pi N_1}{\ve V\ve},\\
\le(\Delta-K\ri)\up_{n+1}&=-\l\me ^{u_0+u_n}G(\me^{u_0+u_n})h(\me^{\up_0+\up_n})-K\up_n+\f{4\pi N_2}{\ve V\ve}.
\end{aligned}
\right.
\ee

\begin{lemma}\label{leo}
 Let $\{(u_n,\up_n)\}$  be the sequence determined by (\ref{uvi}). Then for any lower solution $(u_-,\up_-)$ of (\ref{uvs}), there holds
\be\label{lmr}
\left\{\begin{aligned}
u_1> u_2>\cdot\cdot\cdot> u_n>\cdot\cdot\cdot
>u_{-},\\
\up_1> \up_2>\cdot\cdot\cdot> \up_n>\cdot\cdot\cdot
>\up_{-}.
\end{aligned}
\right.
\ee
Furthermore,  if (\ref{uvi}) has a lower solution, it admits a unique maximal solution $(u_{\l},\up_{\l})$ in the sense that if $(u_{\l}',\up_{\l}')$
is  any other  solution, then $u_{\l}>u_{\l}'$, $\up_{\l}>\up_{\l}'$.
 \end{lemma}
 \begin{proof}
We will prove it by the  induction method.  For $n=1$, by the iteration scheme, we have
 \be\label{uvm}
\left\{\begin{aligned}
\le(\Delta-K\ri)\le( u_2-u_1 \ri)&=4\pi\sum\limits_{j=1}^{N_1}\delta_{p_j'},\\
\le(\Delta-K\ri)\le(\up_2-\up_1\ri)&=4\pi\sum\limits_{j=1}^{N_2}\delta_{p_j^{''}}.
\end{aligned}
\right.
\ee
Then the maximum principle, i.e., Lemma 4.1 in \cite{huang2020existence} indicates $u_2\leq u_1$ and $\up_2\leq \up_1$. Suppose that $u_2-u_1$ attains the maximum 0 at some $x_0\in V$. Then by (\ref{uvm}), we obtain $ \Delta \le( u_2-u_1\ri)(x_0)\geq 0$. However, by  (\ref{mul}), $ \Delta \le( u_2-u_1\ri)(x_0)\leq 0$. Hence,  $(u_2-u_1)(x)=(u_2-u_1)(x_0)=0$ if $x\sim x_0$, which yields $(u_2-u_1)(x)\equiv 0$ since $G$ is connected. This leads to  a contradiction with the inequality $\le(\Delta-K\ri)\le( u_2-u_1 \ri)>0$ at $p_j'$. Therefore, $u_2<u_1$,  and similarly,  $\up_2<\up_1$.
 Now suppose that
 \be
\left\{\begin{aligned}
u_1> u_2>\cdot\cdot\cdot> u_n,\\
\up_1> \up_2>\cdot\cdot\cdot> \up_n.
\end{aligned}
\right.
\ee
Choose $K>\l H(1)G(1)$ . It is seen  from (\ref{uvi}) that
\bna\begin{aligned}
(\D-K)(u_{n+1}-u_n)&=-\l\me ^{\up_0+\up_n}H(\me^{\up_0+\up_n})g(\me^{u_0+u_n})+\l\me ^{\up_0+\up_{n-1}}H(\me^{\up_0+\up_{n-1}})g(\me^{u_0+u_{n-1}})-K(u_n-u_{n-1})\\
&\geq-\l H(1)\le(g(\me^{u_0+u_{n}})-g(\me^{u_0+u_{n-1}})\ri)-K(u_n-u_{n-1})\\
&=\le(\l H(1)e^{\xi}G(\me^{\xi})-K\ri)(u_n-u_{n-1})\\
&\geq \le(\l H(1)G(1)-K\ri)(u_n-u_{n-1})\\
&>0,
\end{aligned}
\ena
where we have used the mean value theorem and $u_0+u_n\leq \xi\leq u_0+u_{n-1}$.  Applying the same method as in proving $u_2<u_1$, we obtain
$u_{n+1}<u_n$. Hence, we get
$$u_1> u_2>\cdot\cdot\cdot> u_n>\cdot\cdot\cdot.$$ Similarly, there also holds
$$\up_1> \up_2>\cdot\cdot\cdot> \up_n>\cdot\cdot\cdot.$$

Next we prove $u_n>u_-$ and $\up_n>\up_-$ for any $n$.  For $n=1$, we derive that 
\be\label{lsi}
\begin{aligned}
\D \le(u_--u_1\ri)&\geq -\l\me ^{\up_0+\up_-}H(\me^{\up_0+\up_-})g(\me^{u_0+u_-})+4\pi\sum\limits_{j=1}^{N_1}\delta_{p_j'}\\
&=-\l\me ^{\up_0+\up_-}H(\me^{\up_0+\up_-})\le[g(\me^{u_0+u_-})-g(\me^{u_0+u_1})\ri]+4\pi\sum\limits_{j=1}^{N_1}\delta_{p_j'}\\
&=\l\me ^{\up_0+\up_-}H(\me^{\up_0+\up_-})\me^{\xi}G(\me^{\xi})(u_--u_1)+4\pi\sum\limits_{j=1}^{N_1}\delta_{p_j'},
\end{aligned}\\
\ee
where $\xi$ lies between $u_--u_1$ and $0$.  Noting that $G$ is finite, we have that there exists $x_0$  such that
$\le( u_--u_1\ri)(x_0)=\max\limits_{x\in V}\le( u_--u_1\ri)(x)$. Assuming that $\le( u_--u_1\ri)(x_0)\geq 0$, then by (\ref{lsi}) we have
 $\D \le(u_--u_1\ri)(x_0)\geq 0$. Again,  we have $\D \le(u_--u_1\ri)(x_0)\leq 0$ by (\ref{mul}). Hence, $(u_--u_1)(x)=(u_--u_1)(x_0)$ if $x\sim x_0$, and $(u_--u_1)(x)\equiv(u_--u_1)(x_0)$ since $G$ is connected, which contradicts (\ref{lsi}) at $p_j'$. Hence, the assumption  is not true and $u_-<u_1$. Similarly, $\up_-<\up_1$.  For some $n\geq 1$, assume that $u_-<u_{n-1}$
 and $\up_-<\up_{n-1}$.  In view of (\ref{lsd}) and (\ref{uvi}), we arrive at 
 \bna\begin{aligned}
(\D-K)(u_{-}-u_n)&\geq -\l\me ^{\up_0+\up_-}H(\me^{\up_0+\up_-})g(\me^{u_0+u_-})+\l\me ^{\up_0+\up_{n-1}}H(\me^{\up_0+\up_{n-1}})g(\me^{u_0+u_{n-1}})-K(u_--u_{n-1})\\
&\geq-\l\me ^{\up_0+\up_{n-1}}H(\me^{\up_0+\up_{n-1}})\le(g(\me^{u_0+u_-})-g(\me^{u_0+u_{n-1}})\ri) -K(u_--u_{n-1})\\
&=\l\me ^{\up_0+\up_{n-1}}H(\me^{\up_0+\up_{n-1}})\me^{\xi}G(\me^{\xi})(u_--u_{n-1})-K(u_--u_{n-1})\\
&\geq \le(\l H(1)G(1)-K\ri)(u_--u_{n-1})\\
&>0,
\end{aligned}
\ena
where $u_-+u_0\leq \xi\leq u_{n-1}+u_0$. By the maximum principle, we have $u_-\leq u_n$. Using the same argument as before, we get $u_-<u_n$.
 Similarly, $\up_-<\up_n$.  
 
It easy to see that if the system (\ref{uvs}) has a lower solution, then it admits a solution $(u_{\l}, \up_{\l})=\lim\limits_{n\ra \infty}(u_n,\up_n)$. If $(u_{\l}',\up_{\l}')$ is any other solution,   noting that $(u_{\l}',\up_{\l}')$ is also a lower solution of (\ref{uvs}), there holds $u_{\l}\geq u_{\l}'$, 
$\up_{\l}\geq \up_{\l}'$. Furthermore, proceeding analogously as before, we get 
\bna
	(\D-K)(u_{\l}'-u_{\l})\geq \le(\l H(1)G(1)-K\ri)(u_{\l}'-u_{\l})\geq 0.
\ena
Assuming that  $\max\limits_{x\in V} (u_{\l}'-u_{\l})(x) =(u_{\l}'-u_{\l})(x_0)=0$ for some $x_0\in V$, then we conclude that 
$\D(u_{\l}'-u_{\l})(x_0)\geq 0$. Hence $(u_{\l}'-u_{\l})(x)=0$ if $x\sim x_0$. The connectedness of $G$ leads to  $(u_{\l}'-u_{\l})(x)\equiv 0$. Similarly, $\up_{\l}'(x)\equiv \up_{\l}(x)$. This contradicts the assumption $(u_{\l}, \up_{\l})\neq (u_{\l}',\up_{\l}')$. Therefore, $u_{\l}>u_{\l}'$, $\up_{\l}>\up_{\l}'$. Thus,  in this sense, $(u_{\l}, \up_{\l})$ is a unique maximal solution.  
\end{proof}
\begin{lemma}
The system (\ref{uvs}) has a solution if $\l$ is big enough.
 \end{lemma}
\begin{proof}
Observe that the functions $u_0$ and $\up_0$ are bounded since $G$ is finite. Thus, there exists $(c_1,c_2)$ such that $u_0-c_1<0$  and $\up_0-c_2<0$.
Let $(u_-,\up_-)=(-c_1,-c_2)$.  It is obvious that
\be
\left\{\begin{aligned}
\Delta u_- &\geq -\l\me ^{\up_0+\up_-}H(\me^{\up_0+\up_-})g(\me^{u_0+u_-})+\f{4\pi N_1}{\ve V\ve},\\
\Delta \up_-&\geq -\l\me ^{u_0+u_-}G(\me^{u_0+u_-})h(\me^{\up_0+\up_-})+\f{4\pi N_2}{\ve V\ve},
\end{aligned}
\right.
\ee
if $\l$ is big enough. Hence $(u_-,\up_-)$ is a lower solution of the system (\ref{uvs}). This guarantees the existence of the solution.
\end{proof}

\begin{lemma}\label{exl}
 There exists $\l_c>0$ such that if $\l>\l_c$,  the system (\ref{uvs}) admits a solution, while if $\l<\l_c$,  the system (\ref{uvs}) admits no solution.
 \end{lemma}
 \begin{proof}
 If the system (\ref{uvs}) admits a solution $(u,\up)$, then by integrating both sides of equations in (\ref{uvs})  on $V$, we get the necessary condition 
\be\label{lst}
\l\geq \f{4\pi\max\{N_1,N_2\}}{G(1)H(1)\ve V\ve }.
\ee
Define the set
$$\Lambda:=\Big\{\l>0\,\big|  \text{ $\l $ is  such  that the system (\ref{uvs})  has a solution}\Big\}.$$
Assume that $\l\in \Lambda$ and denote by $(u_{\l}, \up_{\l})$ the solution to the system (\ref{uvs}). For $\l_1\in \Lambda$ and $\l_1<\l_2$, it follows from (\ref{uvs}) that 
 $(u_{\l_1}, \up_{\l_1})$ is a  lower solution for (\ref{uvs}) with $\l=\l_2$.  Hence, we infer that  $[\l_1,+\infty)\subset \Lambda$ and $\Lambda$ is an interval.   Denote
 $\l_c=\inf\{\l\,|\,\l\in \L\}$. The inequality (\ref{lst}) yields $\l_c\geq \f{4\pi\max\{N_1,N_2\}}{G(1)H(1)\ve V\ve }.$  This completes the proof.
\end{proof}

 Lemma \ref{leo} and Lemma \ref{exl} indicate that if $\l>\l_c$, the system (\ref{uvs}) has a maximal solution.
Denote  by  $\{(u_{\l}, \up_{\l})\,|\,\l>\l_c\}$ the family of maximal solutions of (\ref{uvs}).  Assume  $\l_1>\l_2>\l_c$.   It is easy to check that
\bna\begin{aligned}
\D u_{\l_2}&=-\l_2\me ^{\up_0+\up_{\l_2}}H(\me^{\up_0+\up_{\l_2}})g(\me^{u_0+u_{\l_2}})+\f{4\pi N_1}{\ve V\ve}\\
&= -\l_1\me ^{\up_0+\up_{\l_2}}H(\me^{\up_0+\up_{\l_2}})g(\me^{u_0+u_{\l_2}})+\f{4\pi N_1}{\ve V\ve}\\
&\,\,\,\,\,+(\l_1-\l_2)\me ^{\up_0+\up_{\l_2}}H(\me^{\up_0+\up_{\l_2}})g(\me^{u_0+u_{\l_2}})\\
&\geq -\l_1\me ^{\up_0+\up_{\l_2}}H(\me^{\up_0+\up_{\l_2}})g(\me^{u_0+u_{\l_2}})+\f{4\pi N_1}{\ve V\ve}.
\end{aligned}\ena
Similarly, $$\D \up_{\l_2}\geq -\l_1\me ^{u_0+u_{\l_2}}G(\me^{u_0+u_{\l_2}})h(\me^{\up_0+\up_{\l_2}})+\f{4\pi N_2}{\ve V\ve}.$$
Hence,  $(u_{\l_2}, \up_{\l_2})$ is a lower solution of (\ref{uvs}) with $\l=\l_1$. Thus, $u_{\l_1}\geq u_{\l_2}$and $\up_{\l_1}\geq \up_{\l_2} $ by Lemma  \ref{leo}. Furthermore,  the same argument as before leads to the inequality  
\bna
	\D (u_{\l_2}-u_{\l_1})> \l_1G(1)H(1)(u_{\l_2}-u_{\l_1}).
\ena
 Assuming that $\max\limits_{x\in V}(u_{\l_2}-u_{\l_1})(x)=(u_{\l_2}-u_{\l_1})(x_0)=0$ for some $x_0\in V$. It follows that $\D(u_{\l_2}-u_{\l_1})(x_0)>0$, which is impossible. Hence $u_{\l_1}(x)>u_{\l_2}(x)$ for all $x\in V$. Similarly, $\up_{\l_1}>\up_{\l_2}$.
Next we use priori estimates to deal with the critical case. We make the decomposition $u_{\l}=\bar{u}_{\l}+u'_{\l}$,  where $\bar{u}_{\l}=\f{1}{\ve V\ve}\int_{V}u_{\l} d\mu$ and $u'_{\l}=u_{\l}-\bar{u}_{\l}$.  By (\ref{uvs}), we get
\bna\begin{split}
\V\na u'_{\l}\V_2^2&=\l\iv \me ^{\up_0+\up_{\l}}H(\me^{\up_0+\up_{\l}})g(\me^{u_0+u_{\l}})u'_{\l} d\mu\\
&\leq \l G(1)H(1)\iv \ve u'_{\l} \ve d\mu\leq C\l\ve V\ve^{1/2}\V\na  u'_{\l} \V_2,
\end{split}
\ena
where we have used the Poincar\'{e} inequality, i.e., Lemma 6 in \cite{grigor2016kazdan}. Hence
\be\label{na}
\V\na u'_{\l}\V_2\leq C\l.\ee
Noting $u_0+u_{\l}=u_0+\bar{u}_{\l}+u'_{\l}<0$, by integration on $V$,  we get
\be\label{upb}
\bar{u}_{\l}<-\f{1}{\ve V\ve}\iv u_0(x)d\mu.
\ee
By integrating the second equation in (\ref{uvs}) on $V$, it yields
\be\label{lb}\l \iv \me^{u_0+u_{\l}}d\mu\geq \f{4\pi N_2}{G(1)H(1)}.\ee
 Using   the Trudinger-Moser inequality, i.e., Lemma 7 in \cite{grigor2016kazdan},  we obtain

 \be\label{lo}
\begin{split}
\iv \me^{u_0+u_{\l}}d\mu &= \iv \me^{u_0+\bar{u}_{\l}+u'_{\l}}d\mu\leq \me^{\bar{u}_{\l}}\max\limits_{x\in V}\me^{u_0}\iv \me^{u'_{\l}}d\mu\\
&\leq C \me^{\bar{u}_{\l}}\iv \me^{\V \na u'_{\l}\V_2\f{u'_{\l}}{\V \na u'_{\l}\V_2}}d\mu\leq C \me^{\bar{u}_{\l}}\iv\me^{\V \na u'_{\l}\V_2^2+\f{|u'_{\l}|^2}{4\V \na u'_{\l}\V_2^2}}d\mu\\
&\leq  C \me^{\bar{u}_{\l}}e^{\V\na u'_{\l}\V_2^2}.
\end{split}
\ee
Then (\ref{lb}) and (\ref{lo}) give
$$\me^{\bar{u}_{\l}}\geq C\l^{-1}\me^{-\V \na  u'_{\l}\V_2^2},$$
which together with  (\ref{na}) and (\ref{upb}) gives $$\ve \bar{u}_{\l}\ve\leq C(1+\l+\l^2).$$
Furthermore,
\be\label{pre}
\V u_{\l}\V_{W^{1,2}(V)}\leq C(1+\l+\l^2).\ee
Similarly,
\be\label{prt}
\V \up_{\l}\V_{W^{1,2}(V)}\leq C(1+\l+\l^2).\ee
Set $\l_c< \l<\l_c+1$. Noting (\ref{pre}) and (\ref{prt}) and the fact that the  space  $W^{1,2}(V)$ is precompact, we conclude
$u_{\l}\ra u_*\in W^{1,2}(V)$, $\up_{\l}\ra \up_*\in W^{1,2}(V)$,  pointwisely,
as $\l\ra \l_c$.
Hence,  we deduce that 
$$\D u_{\l}\ra \D u_*,\,\,\D \up_{\l}\ra \D \up_*,$$

$$\l\me ^{\up_0+\up_{\l}}H(\me^{\up_0+\up_{\l}})g(\me^{u_0+u_{\l}}) \ra \l_c\me ^{\up_0+\up_*}H(\me^{\up_0+\up_*})g(\me^{u_0+u_*}),$$
$$\l\me ^{u_0+u_{\l}}G(\me^{u_0+u_{\l}})h(\me^{\up_0+\up_{\l}})\ra \l_c\me ^{u_0+u_*}G(\me^{u_0+u_*})h(\me^{\up_0+\up_*}),$$
as $\l\ra
\l_c$.
Thus,  $(u_*,\up_*)$ is a solution of (\ref{uvs}) with $\l=\l_c$. The  following lemma is established.
\begin{lemma}
If $\l=\l_c$, then the system (\ref{uvs}) admits a solution.
\end{lemma}
Arguing as in proving that $(u_{\l},\up_{\l})$ is monotone, one can show that $ u_{\l}>u_*$ and $\up_{\l}>\up_{*}$ if $\l>\l_c$.

\vskip 20 pt
\noindent{\bf Acknowledgement}

This work is  partially  supported by the National Natural Science Foundation of China (Grant No. 11721101), and by National Key Research and Development Project SQ2020YFA070080.
\vskip 20 pt

\bibliographystyle{elsarticle-num-names}
\bibliography{CHS}

\end{document}